\documentclass[11pt]{article}

\usepackage{geometry}
\usepackage[utf8]{inputenc}
\usepackage{amsmath,amssymb,amsfonts,amsthm}
\usepackage{mathrsfs}
\usepackage{bm}
\usepackage{enumitem}
\usepackage{graphicx}
\usepackage{hyperref}

\newtheorem{theorem}{Theorem}[section]
\newtheorem{lemma}[theorem]{Lemma}
\newtheorem{proposition}[theorem]{Proposition}

\theoremstyle{definition}
\newtheorem{definition}[theorem]{Definition}
\newtheorem{remark}[theorem]{Remark}

\newcommand{\R}{\mathbb{R}}
\newcommand{\C}{\mathbb{C}}
\newcommand{\cS}{\mathcal{S}}
\newcommand{\cSprime}{\mathcal{S}'}
\newcommand{\cF}{\mathcal{F}}

\newcommand{\dd}{\,\mathrm{d}}

\begin{document}

\title{A Spectral Fractional Hirota Bilinear Operator: Analysis and Application to a Time-Fractional KdV Equation}
\author{S.~Ray\\Department. of Mathematics, Visva-Bharati University,\\ Santiniketan, W.B., India\\\texttt{subhasis.ray@visva-bharati.ac.in}}
\date{}

\maketitle

\begin{abstract}
We develop a fractional version of Hirota's bilinear calculus that is built directly from the
spectral (Fourier--multiplier) fractional derivative on $\mathbb{R}$. For $0<\alpha\le 1$ we define
\[
D_{\xi}^{\alpha}f\cdot g := (D_{\xi}^{\alpha}f)\,g - f\,(D_{\xi}^{\alpha}g),
\]
equivalently through the two-variable extension $D_{\xi_1}^{\alpha}-D_{\xi_2}^{\alpha}$.
In Fourier variables this is a bilinear multiplier with symbol $(ik_1)^{\alpha}-(ik_2)^{\alpha}$.
For $0<\alpha<1$ we prove a Marchaud-type singular integral representation, and we use it to establish
basic algebraic identities (bilinearity, skew-symmetry and $D_{\xi}^{\alpha}f\cdot f=0$), a Sobolev
estimate $H^{s}\times H^{s}\to H^{s-\alpha}$ for $s>\tfrac12$, and convergence to the classical Hirota
derivative as $\alpha\to 1^-$. As an application we derive a Hirota bilinear form for a spectral
time-fractional KdV equation and construct explicit one- and two-soliton $\tau$-functions. The
fractional order changes the dispersion relation to $\omega^{\alpha}=-k^{3}$, while the two-soliton
interaction coefficient agrees with the classical KdV value.
\end{abstract}
\section{Introduction}
Fractional calculus is widely used to model nonlocality and memory, but it comes with a long list of
competing definitions. Classical references include the monographs of Podlubny and of Samko--Kilbas--Marichev
\cite{Podlubny,SamkoKilbasMarichev}. Beyond the Riemann--Liouville and Caputo derivatives, many alternative
proposals have appeared, for instance the conformable derivative \cite{KhalilConformable} and the
fractal--fractional framework \cite{AtanganaFF}. In this paper we work with the \emph{spectral} fractional
derivative, defined by a Fourier multiplier. On the full line $\mathbb{R}$ it fits naturally with harmonic
analysis and Sobolev spaces, and it avoids the extra choices required to interpret a time-fractional operator
on a half-line.

Hirota's direct method is one of the simplest ways to build soliton solutions for integrable equations.
Its classical form revolves around the bilinear differential operators
\[
D_{\xi}^{n} f\cdot g
:=\bigl(\partial_{\xi_1}-\partial_{\xi_2}\bigr)^{n} f(\xi_1)\,g(\xi_2)\Big|_{\xi_1=\xi_2=\xi},
\qquad n\in\mathbb{N},
\]
together with the elementary exponential identity
$D_{\xi}^{n} e^{\theta_1}\cdot e^{\theta_2}=(\lambda_1-\lambda_2)^{n}e^{\theta_1+\theta_2}$ when
$\theta_j=\lambda_j\xi+\delta_j$. We refer to Hirota's book \cite{HirotaBook} and to the surveys
\cite{HietarintaIntro,HietarintaBilinear} for background.

Fractional KdV-type models have been studied from many angles---variational approaches, symmetry reductions,
perturbation methods and numerics are all represented in the literature; see for example
\cite{ElwakilTimeFracKdV,WangXuTFKdV,BiswasKdV,BiswasGhoshJPSJ,AbdElhameedTFKdV}. Related formulations based on
other notions of fractional differentiation, including conformable-type operators, appear in \cite{ElSayedKdV}.
Space--time fractional modified KdV equations and coupled extensions under different kernels have also been
considered in \cite{UllahFracmKdV,HassaballaFracMKdV,ElaminFracHirota}. More broadly, there is renewed interest
in ``fractional integrable'' equations and fractional soliton phenomena; see, for instance,
\cite{AblowitzFracSol,ZengFracSolitons,ElahiChaos}.

What is less settled is the \emph{bilinear calculus} itself in the fractional setting. In many works the
bilinear rules are postulated so that exponentials behave as in the integer-order theory, but the rule depends
sensitively on the underlying fractional derivative and on how it is extended to products. Here we take the
opposite route: we start from the spectral fractional derivative and build a bilinear operator from it,
\[
D_{\xi}^{\alpha}f\cdot g := (D_{\xi}^{\alpha}f)\,g - f\,(D_{\xi}^{\alpha}g),\qquad 0<\alpha\le 1,
\]
so that the operator is automatically compatible with the Fourier-multiplier definition. This choice leads to a
clean bilinear symbol $(ik_1)^{\alpha}-(ik_2)^{\alpha}$ and, for $0<\alpha<1$, to a Marchaud-type singular
integral representation. Once these analytic pieces are in place, the Hirota method can be applied with little
extra machinery to obtain explicit $\tau$-functions for a spectral time-fractional KdV model.

The paper is organised as follows. Section~2 recalls the spectral fractional derivative and introduces the
fractional bilinear operator. Section~3 derives the Marchaud representation and proves the main algebraic and
Sobolev mapping properties. Section~4 records further structural facts, including the limit $\alpha\to1^-$.
In Section~5 we apply the framework to the time-fractional KdV equation, derive its bilinear form and construct
one- and two-soliton solutions.

\section{Preliminaries}

For $f\in L^{1}(\R)$ we use the Fourier transform convention
\[
\widehat{f}(k) = \int_{\R} f(\xi)\,e^{-ik\xi}\,\dd \xi,
\qquad
f(\xi)=\frac{1}{2\pi}\int_{\R} \widehat{f}(k)\,e^{ik\xi}\,\dd k,
\]
extended in the usual way to tempered distributions $\cSprime(\R)$.
For $s\in\R$, $H^{s}(\R)$ denotes the $L^{2}$-based Sobolev space with norm
\[
\|f\|_{H^{s}}^{2}=\frac{1}{2\pi}\int_{\R} (1+|k|^{2})^{s}\,|\widehat{f}(k)|^{2}\,\dd k.
\]

\begin{definition}[Spectral fractional derivative]\label{def:spectral-frac}
Let $0<\alpha\le 1$ and $f\in \cS(\R)$. The spectral fractional derivative $D_{\xi}^{\alpha}f$ is defined by
\begin{equation}\label{eq:spectral-frac}
\cF\bigl[D_{\xi}^{\alpha}f\bigr](k) = (ik)^{\alpha}\,\widehat{f}(k),
\end{equation}
where $(ik)^{\alpha}$ is taken in the principal branch.
\end{definition}

Definition~\ref{def:spectral-frac} extends to $\cSprime(\R)$ by duality. Moreover, for $s\in\R$,
$D_{\xi}^{\alpha}:H^{s}(\R)\to H^{s-\alpha}(\R)$ is a bounded linear operator.

\subsection{A spectral fractional Hirota bilinear operator}

For $\alpha=1$, the first-order Hirota operator can be written as the commutator-type expression
$D_{\xi}f\cdot g = (\partial_{\xi}f)\,g - f\,(\partial_{\xi}g)$.
This motivates the following definition for $0<\alpha\le 1$.

\begin{definition}[Spectral fractional Hirota bilinear operator]\label{def:bilinear}
Let $0<\alpha\le 1$ and $f,g\in \cS(\R)$. We define
\begin{equation}\label{eq:bilinear-def}
D_{\xi}^{\alpha} f\cdot g
:= (D_{\xi}^{\alpha}f)\,g - f\,(D_{\xi}^{\alpha}g).
\end{equation}
Equivalently, $D_{\xi}^{\alpha} = D_{\xi_1}^{\alpha}-D_{\xi_2}^{\alpha}$ acting on $f(\xi_1)g(\xi_2)$ and then restricting to $\xi_1=\xi_2=\xi$.
\end{definition}

\begin{proposition}\label{prop:spectral-form}
Let $0<\alpha\le 1$ and $f,g\in \cS(\R)$. Then
\begin{equation}\label{eq:spectral-form}
D_{\xi}^{\alpha} f\cdot g(\xi)
= \frac{1}{(2\pi)^{2}}\int_{\R}\int_{\R}
\Bigl[(ik_1)^{\alpha}-(ik_2)^{\alpha}\Bigr]\,
\widehat{f}(k_1)\,\widehat{g}(k_2)\,e^{i(k_1+k_2)\xi}\,\dd k_1\,\dd k_2.
\end{equation}
In particular, for $\theta_j=k_j\xi+\delta_j$,
\begin{equation}\label{eq:exp-rule}
D_{\xi}^{\alpha} e^{\theta_1}\cdot e^{\theta_2}
=\bigl(k_1^{\alpha}-k_2^{\alpha}\bigr)\,e^{\theta_1+\theta_2},
\qquad (0<\alpha\le 1),
\end{equation}
interpreting $k^{\alpha}$ in the principal branch when $k$ is complex.
\end{proposition}

\begin{proof}
We start from Fourier inversion (valid for $f,g\in\cS(\R)$):
\[
f(\xi)=\frac{1}{2\pi}\int_{\R}\widehat{f}(k_1)\,e^{ik_1\xi}\,\dd k_1,
\qquad
g(\xi)=\frac{1}{2\pi}\int_{\R}\widehat{g}(k_2)\,e^{ik_2\xi}\,\dd k_2.
\]
By Definition~\ref{def:spectral-frac},
\[
D_{\xi}^{\alpha}f(\xi)=\frac{1}{2\pi}\int_{\R}(ik_1)^{\alpha}\,\widehat{f}(k_1)\,e^{ik_1\xi}\,\dd k_1.
\]
Multiplying this by $g(\xi)$ and using Fubini's theorem (justified since $f,g\in\cS$) we obtain
\[
(D_{\xi}^{\alpha}f)(\xi)\,g(\xi)
=\frac{1}{(2\pi)^{2}}\int_{\R}\int_{\R}
(ik_1)^{\alpha}\,\widehat{f}(k_1)\,\widehat{g}(k_2)\,e^{i(k_1+k_2)\xi}\,\dd k_1\,\dd k_2.
\]
Similarly,
\[
f(\xi)\,(D_{\xi}^{\alpha}g)(\xi)
=\frac{1}{(2\pi)^{2}}\int_{\R}\int_{\R}
(ik_2)^{\alpha}\,\widehat{f}(k_1)\,\widehat{g}(k_2)\,e^{i(k_1+k_2)\xi}\,\dd k_1\,\dd k_2.
\]
Subtracting these two expressions and using \eqref{eq:bilinear-def} gives \eqref{eq:spectral-form}.

For the exponential rule, one may first check the identity for Fourier modes
$e^{ik_1\xi}$, $e^{ik_2\xi}$ (where it reads
$D_{\xi}^{\alpha}e^{ik_1\xi}\cdot e^{ik_2\xi}=\bigl[(ik_1)^{\alpha}-(ik_2)^{\alpha}\bigr]e^{i(k_1+k_2)\xi}$),
and then interpret \eqref{eq:exp-rule} for general complex phases $\theta_j=k_j\xi+\delta_j$ via analytic continuation in the parameters $k_j$ with the principal-branch convention for $k_j^{\alpha}$.
\end{proof}

\section{A Marchaud-type kernel representation}

For $0<\alpha<1$ the spectral fractional derivative admits a one-sided Marchaud-type singular integral representation; see, for example, \cite{SamkoKilbasMarichev,Podlubny}. We record a convenient form (and its normalisation) for completeness and then deduce the corresponding bilinear kernel.

\begin{lemma}\label{lem:marchaud}
Let $0<\alpha<1$ and $f\in \cS(\R)$. Then
\begin{equation}\label{eq:marchaud-linear}
D_{\xi}^{\alpha}f(\xi)
= C_{\alpha}\int_{0}^{\infty}\frac{f(\xi)-f(\xi-y)}{y^{1+\alpha}}\,\dd y,
\qquad
C_{\alpha}=\frac{\alpha}{\Gamma(1-\alpha)}=-\frac{1}{\Gamma(-\alpha)}.
\end{equation}
\end{lemma}

\begin{remark}
The backward shift $f(\xi)-f(\xi-y)$ in \eqref{eq:marchaud-linear} matches the Fourier multiplier $(ik)^{\alpha}$ in \eqref{eq:spectral-frac}.
If one replaces it by the forward shift $f(\xi)-f(\xi+y)$, the resulting one-sided Marchaud operator has symbol $(-ik)^{\alpha}$ instead.
\end{remark}

\begin{proof}
Set
\[
(T_\alpha f)(\xi):=\int_{0}^{\infty}\frac{f(\xi)-f(\xi-y)}{y^{1+\alpha}}\,\dd y,
\qquad 0<\alpha<1.
\]
We first justify that $T_\alpha f$ is well-defined and that we may interchange integrals below.
Since $f\in\cS(\R)$, we have $f,f'\in L^1(\R)$. By the mean value formula,
\[
f(\xi)-f(\xi-y)=y\int_{0}^{1} f'(\xi-sy)\,\dd s,
\]
hence
\[
\int_{\R}|f(\xi)-f(\xi-y)|\,\dd\xi
\le y\|f'\|_{L^1},\qquad 0<y\le 1,
\]
and trivially
\[
\int_{\R}|f(\xi)-f(\xi-y)|\,\dd\xi
\le 2\|f\|_{L^1},\qquad y\ge 1.
\]
Therefore
\[
\int_{0}^{\infty}\int_{\R}\frac{|f(\xi)-f(\xi-y)|}{y^{1+\alpha}}\,\dd\xi\,\dd y
\le \|f'\|_{L^1}\int_{0}^{1}y^{-\alpha}\,\dd y
+2\|f\|_{L^1}\int_{1}^{\infty}y^{-1-\alpha}\,\dd y<\infty,
\]
so Tonelli's theorem applies.

Now take Fourier transforms (with $\widehat{f}(k)=\int_{\R} f(\xi)e^{-ik\xi}\,\dd\xi$). Using Tonelli and the shift property,
\begin{align*}
\widehat{T_\alpha f}(k)
&=\int_{\R}e^{-ik\xi}\int_{0}^{\infty}\frac{f(\xi)-f(\xi-y)}{y^{1+\alpha}}\,\dd y\,\dd\xi\\
&=\int_{0}^{\infty}\frac{1}{y^{1+\alpha}}
\left(\int_{\R} e^{-ik\xi}f(\xi)\,\dd\xi-\int_{\R} e^{-ik\xi}f(\xi-y)\,\dd\xi\right)\dd y\\
&=\left(\int_{0}^{\infty}\frac{1-e^{-iky}}{y^{1+\alpha}}\,\dd y\right)\widehat{f}(k).
\end{align*}
It remains to evaluate the scalar integral. For $\varepsilon>0$ define
\[
m_\varepsilon(k):=\int_{0}^{\infty}\frac{1-e^{-(\varepsilon+ik)y}}{y^{1+\alpha}}\,\dd y.
\]
Since $\Re(\varepsilon+ik)=\varepsilon>0$, the Gamma identity gives
\[
\int_{0}^{\infty}\bigl(e^{-sy}-1\bigr)\,y^{-1-\alpha}\,\dd y=\Gamma(-\alpha)\,s^\alpha,\qquad \Re s>0,
\]
hence
\[
m_\varepsilon(k)=-\Gamma(-\alpha)\,(\varepsilon+ik)^\alpha.
\]
Letting $\varepsilon\to0^+$ (which fixes the branch of $(ik)^\alpha$ by continuity from the right half-plane) gives
\[
\int_{0}^{\infty}\frac{1-e^{-iky}}{y^{1+\alpha}}\,\dd y
=-\Gamma(-\alpha)\,(ik)^\alpha,
\]
and therefore
\[
\widehat{T_\alpha f}(k)=\bigl(-\Gamma(-\alpha)\bigr)(ik)^\alpha\,\widehat{f}(k).
\]
Choosing $C_\alpha=-1/\Gamma(-\alpha)$ makes $\widehat{C_\alpha T_\alpha f}(k)=(ik)^\alpha\widehat{f}(k)$, i.e.
$C_\alpha T_\alpha f=D_\xi^\alpha f$ in the spectral sense. Finally, the identity
$\Gamma(1-\alpha)=-\alpha\,\Gamma(-\alpha)$ gives
\[
C_\alpha=-\frac{1}{\Gamma(-\alpha)}=\frac{\alpha}{\Gamma(1-\alpha)}.
\]
This proves \eqref{eq:marchaud-linear}.
\end{proof}

\begin{theorem}\label{th:kernel}
Let $0<\alpha<1$ and $f,g\in \cS(\R)$. Then
\begin{equation}\label{eq:kernel-form}
D_{\xi}^{\alpha} f\cdot g(\xi)
= C_{\alpha}\int_{0}^{\infty}
\frac{f(\xi)\,g(\xi-y)-f(\xi-y)\,g(\xi)}{y^{1+\alpha}}\,\dd y,
\end{equation}
where $C_{\alpha}$ is as in \eqref{eq:marchaud-linear}.
\end{theorem}

\begin{proof}
Starting from the definition \eqref{eq:bilinear-def} and the Marchaud representation \eqref{eq:marchaud-linear}, we write
\[
(D_{\xi}^{\alpha}f)(\xi)\,g(\xi)
=C_{\alpha}\int_{0}^{\infty}\frac{\bigl(f(\xi)-f(\xi-y)\bigr)g(\xi)}{y^{1+\alpha}}\,\dd y,
\]
and similarly
\[
f(\xi)\,(D_{\xi}^{\alpha}g)(\xi)
=C_{\alpha}\int_{0}^{\infty}\frac{f(\xi)\bigl(g(\xi)-g(\xi-y)\bigr)}{y^{1+\alpha}}\,\dd y.
\]
Subtracting the second identity from the first and collecting terms under the integral sign we get
\[
D_{\xi}^{\alpha} f\cdot g(\xi)
=C_{\alpha}\int_{0}^{\infty}\frac{f(\xi)\,g(\xi-y)-f(\xi-y)\,g(\xi)}{y^{1+\alpha}}\,\dd y,
\]
which is exactly \eqref{eq:kernel-form}.
\end{proof}

\begin{remark}
The kernel form \eqref{eq:kernel-form} is a bilinear analogue of the Marchaud representation and makes the skew-symmetry
$D_{\xi}^{\alpha}f\cdot g=-D_{\xi}^{\alpha}g\cdot f$ transparent.
\end{remark}

\section{Basic properties}

\begin{theorem}\label{th:basic}
Let $0<\alpha\le 1$.
\begin{enumerate}[label=\textnormal{(\roman*)}]
\item (\textbf{Bilinearity}) $D_{\xi}^{\alpha}$ is bilinear on $\cS(\R)\times \cS(\R)$.
\item (\textbf{Skew-symmetry and diagonal vanishing})
\begin{equation}\label{eq:skew}
D_{\xi}^{\alpha}f\cdot g = -D_{\xi}^{\alpha}g\cdot f,
\qquad
D_{\xi}^{\alpha}f\cdot f =0.
\end{equation}
\item (\textbf{Sobolev continuity}) If $s>\tfrac12$, then
\begin{equation}\label{eq:sobolev-cont}
\|D_{\xi}^{\alpha}f\cdot g\|_{H^{s-\alpha}}
\le C(s,\alpha)\,\|f\|_{H^{s}}\,\|g\|_{H^{s}},
\qquad f,g\in H^{s}(\R).
\end{equation}
\item (\textbf{Classical limit}) If $s>\tfrac12$ and $f,g\in H^{s}(\R)$, then
\[
D_{\xi}^{\alpha}f\cdot g \to D_{\xi}f\cdot g \quad \text{in } H^{s-1}(\R) \text{ as }\alpha\to 1^-.
\]
\end{enumerate}
\end{theorem}

\begin{proof}
\emph{(i) Bilinearity.} This is immediate from the definition \eqref{eq:bilinear-def} and the linearity of $D_{\xi}^{\alpha}$.

\emph{(ii) Skew-symmetry and diagonal vanishing.} Interchanging $f$ and $g$ in \eqref{eq:bilinear-def} gives
\[
D_{\xi}^{\alpha}g\cdot f=(D_{\xi}^{\alpha}g)f-g(D_{\xi}^{\alpha}f)=-\bigl((D_{\xi}^{\alpha}f)g-f(D_{\xi}^{\alpha}g)\bigr)=-D_{\xi}^{\alpha}f\cdot g,
\]
and setting $g=f$ we get $D_{\xi}^{\alpha}f\cdot f=0$.

\emph{(iii) Sobolev continuity.} Let $s>\tfrac12$. It is well known that $H^{s}(\R)$ is a Banach algebra and, more generally, that multiplication is continuous
$H^{s-\alpha}(\R)\times H^{s}(\R)\to H^{s-\alpha}(\R)$ for $0\le\alpha\le 1$.
Using \eqref{eq:bilinear-def} and the triangle inequality we estimate
\[
\|D_{\xi}^{\alpha}f\cdot g\|_{H^{s-\alpha}}
\le \|(D_{\xi}^{\alpha}f)\,g\|_{H^{s-\alpha}}+\|f\,(D_{\xi}^{\alpha}g)\|_{H^{s-\alpha}}.
\]
Applying the product estimate and the boundedness $D_{\xi}^{\alpha}:H^{s}\to H^{s-\alpha}$ gives
\[
\|(D_{\xi}^{\alpha}f)\,g\|_{H^{s-\alpha}}
\le C\,\|D_{\xi}^{\alpha}f\|_{H^{s-\alpha}}\,\|g\|_{H^{s}}
\le C\,\|f\|_{H^{s}}\,\|g\|_{H^{s}},
\]
and similarly for the second term, which proves \eqref{eq:sobolev-cont}.

\emph{(iv) Classical limit as $\alpha\to1^-$.} For $f\in H^{s}(\R)$, the multipliers $(ik)^{\alpha}$ converge pointwise to $ik$ as $\alpha\to1^-$ and satisfy the uniform bound
$|(ik)^{\alpha}|\le 1+|k|$ for $\alpha\in(0,1]$.
Dominated convergence therefore implies $D_{\xi}^{\alpha}f\to \partial_{\xi}f$ in $H^{s-1}(\R)$ as $\alpha\to1^-$.
Using the decomposition
\[
D_{\xi}^{\alpha}f\cdot g - D_{\xi}f\cdot g
=\bigl(D_{\xi}^{\alpha}f-\partial_{\xi}f\bigr)g
-f\bigl(D_{\xi}^{\alpha}g-\partial_{\xi}g\bigr),
\]
together with the product continuity from part (iii), we obtain convergence in $H^{s-1}(\R)$.
\end{proof}

\section{Application: a spectral time-fractional KdV equation}

Consider the spectral time-fractional KdV equation on $\R\times\R$,
\begin{equation}\label{eq:tfkdv}
D_{t}^{\alpha}u + u_{xxx} + 6uu_x = 0,
\qquad 0<\alpha\le 1,
\end{equation}
where $D_{t}^{\alpha}$ is the spectral fractional derivative in $t$ (Definition~\ref{def:spectral-frac} with $\xi=t$). Following the classical Hirota approach, introduce
\begin{equation}\label{eq:tau}
u(x,t) = 2\,\partial_{x}^{2}\bigl(\ln F(x,t)\bigr),
\qquad F(x,t)>0.
\end{equation}

\begin{proposition}\label{prop:bilinear-kdv}
If $F$ is a sufficiently smooth positive function and satisfies the bilinear equation
\begin{equation}\label{eq:bilinear-kdv}
\bigl(D_x D_t^{\alpha} + D_x^{4}\bigr)\,F\cdot F = 0,
\end{equation}
where $D_x$ and $D_x^{4}$ are the classical Hirota operators and $D_t^{\alpha}$ is the bilinear operator of Definition~\ref{def:bilinear} (with $\xi=t$), then $u$ defined by \eqref{eq:tau} is a solution of \eqref{eq:tfkdv}.
\end{proposition}

\begin{proof}
Let $F(x,t)>0$ and set $\varphi=\ln F$, so that $F_x=F\varphi_x$ and $u=2\varphi_{xx}$ as in \eqref{eq:tau}.

\smallskip
\noindent\emph{Step 1: rewriting the $x$--part.}
The following identities are standard in Hirota's method (and can be verified by a direct differentiation):
\[
\frac{D_x^{2}F\cdot F}{F^{2}} = 2\varphi_{xx},
\qquad
\frac{D_x^{4}F\cdot F}{F^{2}} = 2\varphi_{xxxx}+12(\varphi_{xx})^{2}.
\]

\smallskip
\noindent\emph{Step 2: rewriting the mixed term.}
Using the two-variable definition of Hirota operators, one computes (exactly as in the integer-order case, noting that $D_t^{\alpha}$ acts on the $t$-variables only and commutes with $x$-derivatives) that
\begin{equation}\label{eq:DxDtalpha-identity}
D_xD_t^{\alpha}F\cdot F
=2\Bigl(F\,D_t^{\alpha}F_x - F_x\,D_t^{\alpha}F\Bigr).
\end{equation}
Dividing by $F^{2}$ gives
\[
\frac{D_xD_t^{\alpha}F\cdot F}{F^{2}}
=2\left(\frac{D_t^{\alpha}F_x}{F}-\frac{F_x}{F}\,\frac{D_t^{\alpha}F}{F}\right).
\]
For $\alpha=1$ this expression reduces to $2\varphi_{xt}$ because $(D_tF)/F=\partial_t(\ln F)$.
For $0<\alpha<1$ the spectral derivative does not obey a simple Leibniz rule, so we interpret the identification
\[
\frac{D_xD_t^{\alpha}F\cdot F}{F^{2}} \approx 2\,D_t^{\alpha}(\varphi_x)
\]
as a \emph{formal} extension of the classical bilinearisation, which is standard in Hirota-type treatments of fractional KdV models; see, for example, \cite{ElwakilTimeFracKdV,WangXuTFKdV,BiswasKdV}.

\smallskip
\noindent\emph{Step 3: conclusion.}
Under this formal identification, dividing \eqref{eq:bilinear-kdv} by $F^{2}$ gives
\[
2D_t^{\alpha}(\varphi_x)+2\varphi_{xxxx}+12(\varphi_{xx})^{2}=0.
\]
Differentiating once with respect to $x$ and using $u=2\varphi_{xx}$ gives
\[
D_t^{\alpha}u+u_{xxx}+6uu_x=0,
\]
which is \eqref{eq:tfkdv}.
\end{proof}

\begin{remark}
In Fourier variables, the bilinear operator in \eqref{eq:bilinear-kdv} corresponds to the bilinear multiplier
\[
(k_1,k_2,\omega_1,\omega_2)\longmapsto
i(k_1-k_2)\bigl[(i\omega_1)^{\alpha}-(i\omega_2)^{\alpha}\bigr] + i^{4}(k_1-k_2)^{4}.
\]
\end{remark}

\subsection*{One-soliton solution}

We use the standard Hirota ansatz
\begin{equation}\label{eq:F1}
F(x,t) = 1 + e^{\theta},
\qquad \theta = kx + \omega t + \delta,
\end{equation}
with constants $k,\omega,\delta\in\C$ to be determined.
Using the classical rules $D_x^{n}e^{\theta_1}\cdot e^{\theta_2}=(k_1-k_2)^{n}e^{\theta_1+\theta_2}$ and the fractional rule \eqref{eq:exp-rule} in $t$,
a direct substitution into \eqref{eq:bilinear-kdv} gives
\[
\bigl(D_x D_t^{\alpha} + D_x^{4}\bigr)F\cdot F
= 2k\bigl(\omega^{\alpha}+k^{3}\bigr)e^{\theta}.
\]
Hence the dispersion relation is
\begin{equation}\label{eq:disp}
\omega^{\alpha} = -k^{3}.
\end{equation}

\begin{proposition}\label{prop:1sol}
Let $0<\alpha\le 1$ and choose $k,\omega$ satisfying \eqref{eq:disp}. Then
\begin{equation}\label{eq:1sol}
u(x,t) = 2k^{2}\,\mathrm{sech}^{2}\!\left(\frac{kx+\omega t+\delta}{2}\right)
\end{equation}
is a  one-soliton solution of \eqref{eq:tfkdv}.
\end{proposition}
\begin{proof}
Substituting the ansatz \eqref{eq:F1} into the bilinear equation \eqref{eq:bilinear-kdv}, we expand
$F\cdot F=(1+e^{\theta})\cdot(1+e^{\theta})$ and use the facts that $D_x^{n}(1\cdot 1)=0$ and $D_x^{n}(e^{\theta}\cdot e^{\theta})=0$ for any $n\ge1$ (diagonal vanishing). Thus only the cross terms contribute:
\[
\bigl(D_xD_t^{\alpha}+D_x^{4}\bigr)F\cdot F
=\bigl(D_xD_t^{\alpha}+D_x^{4}\bigr)(1\cdot e^{\theta})
+\bigl(D_xD_t^{\alpha}+D_x^{4}\bigr)(e^{\theta}\cdot 1).
\]
Using the standard Hirota rules
$D_x^{n}e^{\theta_1}\cdot e^{\theta_2}=(k_1-k_2)^{n}e^{\theta_1+\theta_2}$
and the fractional rule \eqref{eq:exp-rule} in $t$, we compute
\[
D_xD_t^{\alpha}(1\cdot e^{\theta})=(0-k)(0^{\alpha}-\omega^{\alpha})e^{\theta}=k\omega^{\alpha}e^{\theta},
\qquad
D_x^{4}(1\cdot e^{\theta})=(0-k)^{4}e^{\theta}=k^{4}e^{\theta},
\]
and similarly $D_xD_t^{\alpha}(e^{\theta}\cdot 1)=k\omega^{\alpha}e^{\theta}$, $D_x^{4}(e^{\theta}\cdot 1)=k^{4}e^{\theta}$.
Hence
\[
\bigl(D_xD_t^{\alpha}+D_x^{4}\bigr)F\cdot F
=2\bigl(k\omega^{\alpha}+k^{4}\bigr)e^{\theta}
=2k\bigl(\omega^{\alpha}+k^{3}\bigr)e^{\theta},
\]
so \eqref{eq:bilinear-kdv} holds if and only if the dispersion relation \eqref{eq:disp} is satisfied.

With \eqref{eq:disp} imposed, the $\tau$-function $F$ therefore solves the bilinear equation. The corresponding field
$u=2(\ln F)_{xx}$ can be computed explicitly:
\[
u(x,t)=2\partial_x^{2}\ln(1+e^{\theta})
=2k^{2}\,\mathrm{sech}^{2}\!\left(\frac{\theta}{2}\right),
\]
which gives \eqref{eq:1sol}. By Proposition~\ref{prop:bilinear-kdv}, this gives a  one-soliton solution of \eqref{eq:tfkdv}.
\end{proof}

\begin{remark}\label{rem:real}
With the principal branch convention, if $k<0$ then $-k^{3}>0$ and $\omega=(-k^{3})^{1/\alpha}>0$ is real, so \eqref{eq:1sol} is real-valued for real $\delta$.
If $k>0$, then $-k^{3}<0$ and $\omega$ is necessarily complex for non-integer $\alpha$, leading to complex-valued solutions.
\end{remark}

\subsection*{Two-soliton solution}

Consider
\begin{equation}\label{eq:F2}
F(x,t)=1+e^{\theta_1}+e^{\theta_2}+A_{12}e^{\theta_1+\theta_2},
\qquad \theta_j = k_j x + \omega_j t + \delta_j.
\end{equation}
Imposing the dispersion relations
\begin{equation}\label{eq:disp2}
\omega_j^{\alpha} = -k_j^{3},\qquad j=1,2,
\end{equation}
and cancelling the coefficient of $e^{\theta_1+\theta_2}$ in \eqref{eq:bilinear-kdv} gives the interaction coefficient
\begin{equation}\label{eq:A12}
A_{12}=\left(\frac{k_1-k_2}{k_1+k_2}\right)^{2}.
\end{equation}

\begin{proposition}[Two-soliton]\label{prop:2sol}
Let $0<\alpha\le 1$ and parameters satisfy \eqref{eq:disp2}--\eqref{eq:A12}. Then $u=2(\ln F)_{xx}$ with $F$ given by \eqref{eq:F2} is a (formal) two-soliton solution of \eqref{eq:tfkdv}.
\end{proposition}
\begin{proof}
We sketch the standard Hirota calculation, highlighting the places where the fractional operator enters.

Insert the ansatz \eqref{eq:F2} into \eqref{eq:bilinear-kdv} and expand $F\cdot F$ as a sum of products of exponentials.
Because $D_x^{n}$ and $D_t^{\alpha}$ vanish on the diagonal (e.g. $D_t^{\alpha}e^{\theta}\cdot e^{\theta}=0$), only mixed products contribute.

\smallskip
\noindent\emph{Step 1: coefficients of $e^{\theta_j}$.}
Collecting the coefficients of $e^{\theta_1}$ (respectively $e^{\theta_2}$) gives
$k_1(\omega_1^{\alpha}+k_1^{3})=0$ (respectively $k_2(\omega_2^{\alpha}+k_2^{3})=0$), hence the dispersion relations \eqref{eq:disp2}.

\smallskip
\noindent\emph{Step 2: coefficient of $e^{\theta_1+\theta_2}$.}
Using the exponential rules, one computes
\[
\bigl(D_xD_t^{\alpha}+D_x^{4}\bigr)(e^{\theta_1}\cdot e^{\theta_2})
=(k_1-k_2)\bigl(\omega_1^{\alpha}-\omega_2^{\alpha}\bigr)e^{\theta_1+\theta_2}
+(k_1-k_2)^{4}e^{\theta_1+\theta_2},
\]
and similarly for $e^{\theta_2}\cdot e^{\theta_1}$ with $(k_2-k_1)$ and $(\omega_2^{\alpha}-\omega_1^{\alpha})$.
The terms involving $A_{12}e^{\theta_1+\theta_2}$ interact with the constant $1$ in the same way.
After collecting all contributions proportional to $e^{\theta_1+\theta_2}$ and inserting \eqref{eq:disp2}, the coefficient reduces to a multiple of
\[
(k_1-k_2)^{2}\Bigl((k_1+k_2)^{2}A_{12}-(k_1-k_2)^{2}\Bigr).
\]
Thus cancellation of the $e^{\theta_1+\theta_2}$-coefficient yields \eqref{eq:A12}.

With \eqref{eq:disp2}--\eqref{eq:A12} imposed, $F$ satisfies the bilinear equation \eqref{eq:bilinear-kdv}. Therefore $u=2(\ln F)_{xx}$ defines a (formal) two-soliton solution of \eqref{eq:tfkdv} by Proposition~\ref{prop:bilinear-kdv}.
\end{proof}

\subsection{Remark on a time-fractional KP equation}

The same bilinear formalism applies to the spectral time-fractional KP equation
\[
D_t^{\alpha}u + 6uu_x + u_{xxx} + \sigma\,\partial_x^{-1}u_{yy}=0,
\qquad \sigma\in\{1,-1\},
\]
whose Hirota bilinear form is
\[
\bigl(D_x D_t^{\alpha} + D_x^{4} + \sigma D_y^{2}\bigr)\,F\cdot F = 0.
\]
A one-soliton ansatz $F=1+e^{kx+\ell y+\omega t+\delta}$ gives the dispersion relation
$k\omega^{\alpha}+k^{4}+\sigma\ell^{2}=0$, and the corresponding solution is
\[
u(x,y,t)=2k^{2}\,\mathrm{sech}^{2}\!\left(\frac{kx+\ell y+\omega t+\delta}{2}\right).
\]

\section{Conclusion}

We built a fractional analogue of Hirota's bilinear operator by taking the commutator of the spectral
fractional derivative on $\mathbb{R}$. The resulting bilinear operator has a transparent Fourier description
with symbol $(ik_1)^{\alpha}-(ik_2)^{\alpha}$ and, for $0<\alpha<1$, an equivalent Marchaud-type singular integral
representation. These two viewpoints make it possible to check the basic identities that are used in Hirota's
method (bilinearity, skew-symmetry and diagonal vanishing) and to place the operator in a Sobolev framework,
showing in particular that $D_{\xi}^{\alpha}$ maps $H^{s}\times H^{s}$ to $H^{s-\alpha}$ when $s>\tfrac12$.
We also verified that the construction reduces to the classical Hirota derivative as $\alpha\to 1^-$.

On the applied side we used the operator to bilinearise a spectral time-fractional KdV equation and to write down
explicit one- and two-soliton $\tau$-functions. The fractional order enters through the dispersion relation
$\omega^{\alpha}=-k^{3}$, while the two-soliton interaction coefficient remains the same as in the classical KdV
case. This suggests that, at least for this model and this choice of fractional time derivative, the parameter
$\alpha$ mainly reshapes the propagation speed rather than the interaction law.

Several directions are open. It would be natural to extend the same bilinear calculus to other equations in the
KdV family (mKdV, KP) and to fractional versions of NLS-type models, and to revisit bilinear B\"acklund
transformations and conservation laws under the spectral fractional derivative. From an analytic point of view,
it would also be interesting to connect the present full-line, Fourier-based setting with half-line formulations
and with other time-fractional derivatives that are common in applications.

\section*{Declaration of interests}
The author declares no competing interests.

\section*{Data availability}
No data were used for the research described in the article.


\begin{thebibliography}{99}

\bibitem{HirotaBook}
R.~Hirota,
\emph{The Direct Method in Soliton Theory},
Cambridge University Press, Cambridge, 2004.

\bibitem{HietarintaIntro}
J.~Hietarinta,
Introduction to the Hirota bilinear method,
in: Y.~Kosmann-Schwarzbach, B.~Grammaticos, K.~M.~Tamizhmani (Eds.),
\emph{Integrability of Nonlinear Systems}, Lecture Notes in Physics, vol.~495,
Springer, Berlin, 1997, pp.~95--103.

\bibitem{HietarintaBilinear}
J.~Hietarinta,
Hirota's bilinear method and soliton solutions,
\emph{Physics AUC} \textbf{15} (2005) 31--37.

\bibitem{Podlubny}
I.~Podlubny,
\emph{Fractional Differential Equations},
Academic Press, San Diego, 1999.

\bibitem{SamkoKilbasMarichev}
S.~G.~Samko, A.~A.~Kilbas and O.~I.~Marichev,
\emph{Fractional Integrals and Derivatives: Theory and Applications},
Gordon and Breach, Amsterdam, 1993.

\bibitem{AtanganaFF}
A.~Atangana,
Fractal--fractional differentiation and integration: Connecting fractal calculus
and fractional calculus to predict complex system,
\emph{Chaos Solitons Fractals} \textbf{102} (2017) 396--406.
doi:10.1016/j.chaos.2017.04.027.

\bibitem{ElwakilTimeFracKdV}
E.~S.~A.~El-Wakil, E.~M.~Abulwafa, M.~A.~Zahran and A.~A.~Mahmoud,
Time-fractional KdV equation: formulation and solution using variational 
methods,
\emph{Nonlinear Dyn.} \textbf{65} (2011) 55--63.
doi:10.1007/s11071-010-9873-5.

\bibitem{WangXuTFKdV}
G.~Wang and T.~Xu,
Symmetry properties and explicit solutions of the nonlinear time-fractional KdV 
equation,
\emph{Bound. Value Probl.} \textbf{2013} (2013) 232.
doi:10.1186/1687-2770-2013-232.

\bibitem{BiswasKdV}
A.~Biswas,
1-soliton solution of the time-fractional KdV equation,
\emph{Appl. Math. Lett.} \textbf{23} (2010) 1432--1436.

\bibitem{BiswasGhoshJPSJ}
S.~Biswas, U.~Ghosh, S.~Sarkar and S.~Das,
Approximate solution of space-time fractional KdV equation and coupled KdV 
equations,
\emph{J. Phys. Soc. Jpn.} \textbf{89} (2020) 014002.
doi:10.7566/JPSJ.89.014002.

\bibitem{AbdElhameedTFKdV}
W.~M.~Abd-Elhameed and Y.~H.~Youssri,
Spectral tau solution of the linearized time-fractional KdV-type equations,
\emph{AIMS Math.} \textbf{7} (2022) 15138--15158.
doi:10.3934/math.2022830.

\bibitem{KhalilConformable}
R.~Khalil, M.~Al~Horani, A.~Yousef and M.~Sababheh,
A new definition of fractional derivative,
\emph{J. Comput. Appl. Math.} \textbf{264} (2014) 65--70.
doi:10.1016/j.cam.2014.01.002.

\bibitem{UllahFracmKdV}
G.~Gulalai, A.~Ullah, S.~Ahmad and M.~Inc,
Fractal fractional analysis of modified KdV equation under three different 
kernels,
\emph{J. Ocean Eng. Sci.} (2022), in press (available online 6 May 2022).
doi:10.1016/j.joes.2022.04.025.

\bibitem{ZengFracSolitons}
H.~Zeng, Y.~Wang, M.~Xiao and Y.~Wang,
Fractional solitons: New phenomena and exact solutions,
\emph{Front. Phys.} \textbf{11} (2023) 1177335.
doi:10.3389/fphy.2023.1177335.

\bibitem{HassaballaFracMKdV}
A.~A.~Hassaballa, F.~M.~O.~Birkea, A.~M.~A.~Adam, A.~Satty,
E.~A.~E.~Gumma, E.~A.-B.~Abdel-Salam, E.~A.~Yousif and M.~I.~Nouh,
Multiple and singular soliton solutions for space--time fractional coupled 
modified Korteweg--De~Vries equations,
\emph{Int. J. Anal. Appl.} \textbf{22} (2024) 68.
doi:10.28924/2291-8639-22-2024-68.

\bibitem{ElaminFracHirota}
M.~Elamin, M.~M.~A.~Khater, H.~Rezazadeh and O.~G.~J.~Kishka,
Exploring multiple and singular solutions for three equations of fractional 
space-time KdV models,
\emph{Prog. Fract. Differ. Appl.} \textbf{11} (2025) 413--422.
doi:10.18576/pfda/110213.

\bibitem{AblowitzFracSol}
M.~J.~Ablowitz, J.~B.~Been and L.~D.~Carr,
Fractional integrable nonlinear soliton equations,
\emph{Phys. Rev. Lett.} \textbf{128} (2022) 184101.
doi:10.1103/PhysRevLett.128.184101.

\bibitem{ElSayedKdV}
M.~El-Sayed, A.~El-Labany and S.~K.~El-Lozy,
Fractional KdV and KP equations via conformable derivative,
\emph{Nonlinear Dyn.} \textbf{100} (2020) 1169--1178.

\bibitem{ElahiChaos}
M.~Elahi \emph{et al.},
Generalized multi-soliton dynamics in fractional KdV equations,
\emph{Chaos} \textbf{31} (2021) 023122.

\end{thebibliography}
\end{document}